\documentclass[a4paper]{amsart}
\usepackage{amsthm,amsmath,amssymb,mathrsfs}
\usepackage[latin1]{inputenc}
 \newtheorem{thm}{Theorem}
 
 \newtheorem{lemma}[thm]{Lemma}

 \theoremstyle{definition}
 \newtheorem{definition}[thm]{Definition}
 \newtheorem{ex}[thm]{Example}
\newtheorem{conj}[thm]{Conjecture}

 \theoremstyle{remark}
 \newtheorem{remark}[thm]{Remark}
\numberwithin{thm}{section}

\def\Spec{{\rm Spec}\,}

\def\GL{{\rm GL}}
\def\l{\text{length}}
\def\PGL{{\rm PGL}}
\def\GSp{{\rm GSp}}

\def\rk{{\rm rk}}
\def\defect{{\rm def}}

\def\dom{{\rm dom}}

\def\Hom{{\rm Hom}}
\def\dim{{\rm dim}}

\def\defect{{\rm def}}

\def\e{\epsilon}
\def\dl{(\!(}
\def\dr{)\!)}
\def\ll{[\![}
\def\rr{]\!]}
\def\D{{\mathscr D}}

\def\N{{\mathcal N}}
\def\R{{\mathbb R}}
\def\Z{{\mathbb Z}}
\def\F{{\mathbb F}}
\def\Q{{\mathbb Q}}
\def\A{{\mathcal A}}
\def\G{{\mathcal G}}
\def\O{{\mathcal O}}
\def\S{{\mathcal S}}
\begin{document}

\begin{title}
{On the geometry of the Newton stratification}
\end{title}
\author{Eva Viehmann}
\address{Technische Universit\"at M\"unchen\\Fakult\"at f\"ur Mathematik - M11 \\ Boltzmannstr. 3\\85748 Garching bei M\"unchen\\Germany}
\email{viehmann@ma.tum.de}
\date{}
\thanks{The author was partially supported by ERC starting grant 277889 ``Moduli spaces of local $G$-shtukas''.}

\begin{abstract}{We give an overview over recent results on the global structure and geometry of the Newton stratification of the reduction modulo $p$ of Shimura varieties of Hodge type with hyperspecial level structure. More precisely, we discuss non-emptiness, dimensions, and closure relations of Newton strata. We also explain the group-theoretic description and methods leading to their proofs.}
\end{abstract}
\maketitle
\section{Introduction}\label{intro}
One of the key invariants of an abelian variety in characteristic $p$ is the Newton polygon describing the isogeny class of its $p$-divisible group. It is a central tool in the study of the reduction modulo $p$ of Shimura varieties of PEL type in almost all of the main breakthroughs in this area (as for example the computation of the Hasse-Weil $\zeta$-function by Kottwitz in \cite{Kottwitzmodp}, or the proof of the local Langlands correspondence for $\GL_n$ by Harris and Taylor \cite{HarrisTaylor}).

However, even very basic questions on the geometry of the induced stratification of the fiber at $p$ of a given Shimura variety of PEL type remained open. For example, except for special cases, one did not even know the set of strata. For several years there is now a trend towards applying more group-theoretic methods (building on the classical works of Kottwitz \cite{Kottwitz1}) to address these questions. In this overview we report on recent developments in this direction. We describe the set of strata, discuss  their dimensions and the closure relations. The group-theoretic methods underlying the proofs do not make a difference between groups associated with Shimura varieties of PEL type or not. Therefore, the natural context to apply them is the most general one where we still have a good theory of integral models of the Shimura varieties and a translation between points in the special fiber and suitable elements of the corresponding group. Thus, instead of Shimura varieties of PEL type we also discuss the recent generalizations to Hodge type Shimura varieties. Besides, we briefly report on the parallel theory for the function field case, i.e. for moduli spaces of global and local $G$-shtukas, or for loop groups.

\section{Points in the reduction of Shimura varieties of Hodge type}
\subsection{The Siegel case}
Let $\mathcal{A}_g$ be the moduli space of principally polarized abelian varieties of dimension $g$ over $\Spec(\mathbb{F}_p)$ for some fixed characteristic $p>0$. It is the fiber at $p$ of a Shimura variety for the group $\GSp_{2g}$. Let $k$ be an algebraically closed field of characteristic $p$. Then the points of $\mathcal{A}_g(k)$ correspond to pairs $(A,\lambda)$ where $A$ is an abelian variety over $k$ and $\lambda$ a principal polarization of $A$. An important invariant of an abelian variety in characteristic $p$ is its $p$-divisible group $A[p^{\infty}]$. It is equipped with a quasi-polarization $\lambda$ induced by $\lambda$ on $A$. By Dieudonn\'e theory this datum corresponds bijectively to the Dieudonn\'e module $(M,F,\langle\cdot,\cdot\rangle)$ where $M$ is a free $W(k)$-module of rank $2g$, where $\langle\cdot,\cdot\rangle$ is a symplectic pairing induced by $\lambda$ and where $F$ is a Frobenius-linear map $M\rightarrow M$ with $pM\subset F(M)$, and satisfying a compatibility condition with $\lambda$. Trivializing $M$ in such a way that $\langle\cdot,\cdot\rangle$ is identified with the standard pairing we can write $F=b\sigma$ for some $b\in \GSp_{2g}(W(k)[1/p])$ and where $\sigma$ is the Frobenius of $W(k)[1/p]$ over $\mathbb{Q}_p$. The element $b$ is well-defined up to base change, i.e. up to replacing it by $g^{-1}b\sigma(g)$ for some $g\in \GSp_{2g}(W(k))$.

A similar, but more tedious description is available for all Shimura varieties of PEL type with good reduction at $p$, compare \cite{vw}, 1 and 7.

\subsection{Shimura varieties of Hodge type}\label{sec23}

Let now $\mathcal{D}=(G,X)$ be a Shimura datum of Hodge type, and let $p$ be a prime. Let $K=K^pK_p\subset G(\mathbb{A}_f^p)G(\mathbb{Q}_p)$ be a compact open subgroup. We always assume that we are in the case of good reduction, i.e. that $K_p$ is hyperspecial. In other words we assume that $G$ extends to a reductive group over $\mathbb{Z}_p$ and that $K_p=G_{\mathbb{Z}_p}(\mathbb{Z}_p)$. This implies in particular that $G$ is unramified at $p$, i.e.~quasi-split and split over an unramified extension of $\mathbb{Q}_p$.

Then the corresponding Shimura variety $Sh_K(G,X)$ is a moduli space of abelian varieties with certain Hodge cycles associated with the Shimura datum. By \cite{KisinIM} it has an integral model at the prime $p$ which is obtained by taking an embedding into a suitable Siegel moduli space, and taking the normalization of the closure in an integral model of that space. We denote its special fiber by $\S_K(G,X)$. Associated with closed points of $\S_K(G,X)$ we thus still have an abelian variety, but there is no explicit moduli theoretic interpretation of $\S_K(G,X)$.

Let $h:\mathbb{S}\rightarrow G_{\mathbb{R}}$ be an element of the conjugacy class $X$ of the Shimura datum. Let $\psi=\psi_h$ be the cocharacter defined on $R$-points (for any $\mathbb{C}$-algebra $R$) by $$R^{\times}\rightarrow (R\times c^*(R))^{\times}=(R\otimes_{\mathbb R}\mathbb C)^{\times}=\mathbb{S}(R)\rightarrow G(R).$$ Here $c$ denotes complex conjugation. Let $\mu=\sigma(\psi^{-1})$. Note that our use of the letter $\mu$ differs slightly from that of \cite{Kisin}, whom we follow for the next construction.

Let $k$ be an algebraically closed field of characteristic $p$ and let $\O_L=W(k)$. Let $x\in \S_K(G,X)(k)$. Let $\A_x$ be the fiber at $x$ of the universal abelian variety $\A$, and let $\G_x$ be its $p$-divisible group. Let $g=\dim \A$. Following \cite{Kisin}, 1.4.1 (or \cite{vw}, 7.2 for the PEL case) there is a trivialization of the Dieudonn\'e module $\mathbb{D}(\G_x)(\O_L)$ of $\G_x$, i.e. an isomorphism $$\O_L^{2g}\rightarrow \mathbb{D}(\G_x)(\O_L).$$ and such that it maps the Hodge tensors on $\A_x$ to certain fixed $\phi$-invariant tensors  defining a subgroup $G_{\O_L}\cong G_{\Z_p}\otimes_{\Z_p}W\subseteq \GL_{2g,\O_L}$. By fixing one such isomorphism we can write $\phi=b\sigma$ where $\sigma$ is the Frobenius on $\O_L$ and where $b\in G_{\mathbb{Z}_p}(\O_L)\mu(p)G_{\mathbb{Z}_p}(\O_L)\subset  G(L)$. The element $b$ is well-defined up to $\sigma$-conjugation by elements of $G_{\mathbb{Z}_p}(\O_L)$. For the Siegel case, this coincides with the class constructed in the previous paragraph.

For $c\in G(L)$ let $$\ll c\rr:=\{g^{-1}c\sigma(g)\mid g\in G(\O_L)\}.$$
Let $C(G,\mu)=\{\ll c\rr \mid c\in G(\O_L)\mu(p)G(\O_L)\}.$ Then by the considerations above we have a natural map $$\Upsilon:\S_K(G,X)(k)\rightarrow C(G,\mu).$$

\section{The set of Newton points}
Before introducing the Newton stratification on the special fiber of a Shimura variety we discuss in this section the index set for such stratifications. This set of Newton points has an abstract group-theoretic definition independent of the context of the previous section. Thus for this section we are in the following, more general setting. 

Let $F$ be a local field of residue characteristic $p$, let $\O_F$ be its ring of integers, and $\e$ a uniformizer. Let $L\supset \O_L$ be the completion of the maximal unramified extension and its ring of integers, and let $\sigma$ denote the Frobenius of $L$ over $F$.

Let $G$ be a reductive group over $\O_F$, and fix a Borel subgroup $B$ and a maximal torus $T$, both defined over $F$.

\subsection{$\sigma$-conjugacy classes and their Newton points}

For $b\in G(L)$ let $$[b]=\{g^{-1}b\sigma(g)\mid g\in G(L)\}$$ denote its $\sigma$-conjugacy class, and let $B(G)$ be the set of $\sigma$-conjugacy classes for $b\in G(L)$. Kott\-witz \cite{Kottwitz1} has given a classification of $B(G)$ generalizing the Dieudonn\'e-Manin classification of isocrystals by Newton polygons. He assigns to an element of $B(G)$ two invariants. One, the Newton point, is the direct analog of the Newton polygon. It is defined by assigning to every $b\in G(L)$ a homomorphism $\nu_b:\mathbb{D}\rightarrow G$. Here $\mathbb D$ is the pro-algebraic torus with character group $\Q$. Varying $b$ in $[b]$ leads to conjugate homomorphisms. The $G(L)$-conjugacy class of $\nu_b$ is stable under the action of $\sigma$. The Newton point associated with $[b]$ is then defined to be this conjugacy class, an element $\nu([b])\in (G\backslash X_*(G)_{\Q})^{\Gamma}$. Often, one represents it by its unique representative in $\N(G):=(W\backslash X_*(T)_{\Q})^{\Gamma}$ where $W$ is the Weyl group of $G$. Choosing dominant representatives we can also identify $\N(G)$ with $X_*(T)_{\Q,\dom}^{\Gamma}$. We call $\N(G)$ the Newton cone, and denote the map $B(G)\rightarrow \N(G)$ by $\nu$.

The second classifying invariant of $[b]$ is the image under the so-called Kottwitz map  $$\kappa_G:B(G)\rightarrow \pi_1(G)_{\Gamma}$$
(see also Rapoport and Richartz \cite{RapoportRichartz} for the reformulation). Here, $\pi_1(G)$ is Borovoi's fundamental group, defined as the quotient of $X_*(T)$ by the coroot lattice, and we are taking coinvariants under the absolute Galois group of $F$. As $G$ is an unramified reductive group over $F$, the map $\kappa_G$ has the following explicit description. By the Cartan decomposition every $b\in G(L)$ is in $G(\O_L)\mu(\e)G(\O_L)$ for some $\mu\in X_*(T)_{\dom}$. Then $\kappa_G([b])$ is the image of $\mu$ under the projection map to $\pi_1(G)_{\Gamma}$.

The images of $\nu([b])$ and $\kappa_G([b])$ in $\pi_{1}(G)_{\Gamma}\otimes_{\Z}\Q$ coincide. Thus for groups where $\pi_{1}(G)_{\Gamma}$ is torsion free, the Newton point alone already determines an element of $B(G)$.

\begin{ex}
\begin{enumerate}
\item For the group $\GL_n$ we choose the upper triangular matrices as Borel subgroup and let $T$ be the diagonal torus. Then we have $X_*(T)_{\mathbb{Q}}\cong \Q^n$ and $\nu=(\nu_i)$ is dominant if and only if $\nu_i\geq \nu_{i+1}$ for all $i$. The Newton point $\nu=(\nu_i)$ of $[b]$ coincides with the classical Newton point of the isocrystal $(L^n,b\sigma)$ for any representative $b\in [b]$. Let $p_{\nu}$ be the polygon associated with $\nu$, that is the graph of the continuous, piecewise linear function $[0,n]\rightarrow \R$ mapping $0$ to $0$ and with slope $\nu_i$ on $[i-1,i]$. Then $\kappa_G([b])\in\pi_1(G)_{\Gamma}\cong \Z$ is equal to $\sum \nu_i$, i.e.~to the second coordinate of the endpoint of the Newton polygon.
\item Let us now consider the group $\PGL_n$. Here we have $X_*(T)\cong \Z^n/\Z\cdot(1,\dotsc,1)$. Let $b_1\in \PGL_n(L)$ with $b_1(e_i)=e_{i+1}$ if $i<n$, and $b_1(e_n)=\e e_1$. Then $(b_1\sigma)^n=b_1^n\sigma^n=\epsilon\sigma^n$. Hence for every integer $l$ the Newton point of $b_1^l$ is constant and thus equal to $(0,\dotsc,0)$ in $X_*(T)_{\mathbb Q}$. However, we have $\pi_1(G)_{\Gamma}\cong \Z/n\Z$ and $\kappa_{\PGL_n}(b_1^l)\equiv l\pmod n$. Thus $[b_1^l]=[1]$ if and only if $n|l$.
\end{enumerate} 
\end{ex}

Let $\N(G)_{\Z}\subset \N(G)$ be the image of $\nu$. It is called the Newton lattice. This set also has an intrinsic, group-theoretic description, see \cite{Chai}. For our standard example $G=\GL_n$ we obtain that $\nu\in\N(\GL_n)_{\Z}$ if and only if all break points of $p_{\nu}$ as well as its endpoint have integral coordinates.

\subsection{$B(G)$ as ranked poset}
In this section we explain Chai's theory of the ordering and lengths of chains in the set of Newton points, \cite{Chai}.

There is a natural partial ordering on the Newton lattice $\N(G)_{\Z}$, induced by the partial ordering on $X_*(T)_{\Q,\dom}$ and given by $\nu\preceq \nu'$ if and only if $\nu'-\nu$ is a non-negative rational linear combination of positive coroots. A characterization not using our choice of $B$ (and thus showing that it is independent of that choice) is the following. Let $\nu,\nu'\in W\backslash X_*(T)_{\Q}$ and $W\cdot\nu, W\cdot \nu'$ the corresponding orbits in $X_*(T)_{\R}$. Then $\nu\preceq\nu'$ if and only if the convex hull of $W\cdot \nu'$ contains the convex hull of $W\cdot \nu$.
On the set $B(G)$ this induces a partial ordering given by $[b']\preceq [b]$ if and only if $\nu([b'])\preceq\nu([b])$ and $\kappa_G([b])=\kappa_G([b'])$.

For the group $\GL_n$ we have $(\nu_i)=\nu\preceq \nu'=(\nu_i')$ if and only if $\sum_{i=1}^l \nu_i\leq \sum_{i=1}^l \nu_i'$ for all $l$ with equality for $l=n$. For the associated polygons this means that the polygon for $\nu'$ lies on or above the polygon of $\nu$ and that they have the same endpoints. Figure \ref{alpha} shows the ordering on the set of $\nu\in \N(\GL_4)_{\mathbb Z}$ with $\nu\preceq (1,1,0,0)$. The arrows point from larger to smaller elements. For better readability we include the relevant points with integral coordinates lying on or below the polygons. An additional important combinatorial observation then is the following: Each polygon is the convex hull of the set of integral points lying on or below the polygon, and for neighboring polygons, these sets differ by one element which is a breakpoint of the larger of the two polygons. (Compare also Example \ref{exlengthgln} below.)

{\setlength{\unitlength}{10 pt}
\begin{figure}[h]
\begin{center}
\begin{picture}(28,8)(0,1)

\put(0,4){\line(1,1){2}}
\put(2,6){\line(1,0){2}}
\put(0,4){\makebox(0,0){$\bullet$}}
\put(2,5){\makebox(0,0){$\bullet$}}
\put(4,6){\makebox(0,0){$\bullet$}}
\put(1,5){\makebox(0,0){$\bullet$}}
\put(2,6){\makebox(0,0){$\bullet$}}
\put(3,6){\makebox(0,0){$\bullet$}}
\put (2,3){\makebox(0,0){$(1,1,0,0)$}}

\put(8,4){\line(1,1){1}}
\put(11,6){\line(1,0){1}}
\put(9,5){\line(2,1){2}}
\put(8,4){\makebox(0,0){$\bullet$}}
\put(10,5){\makebox(0,0){$\bullet$}}
\put(12,6){\makebox(0,0){$\bullet$}}
\put(9,5){\makebox(0,0){$\bullet$}}
\put(11,6){\makebox(0,0){$\bullet$}}
\put (10,3){\makebox(0,0){$(1,\frac{1}{2},\frac{1}{2},0)$}}

\put(15.5,7){\line(1,1){1}}
\put(16.5,8){\line(3,1){3}}
\put(15.5,7){\makebox(0,0){$\bullet$}}
\put(17.5,8){\makebox(0,0){$\bullet$}}
\put(19.5,9){\makebox(0,0){$\bullet$}}
\put(16.5,8){\makebox(0,0){$\bullet$}}
\put (17.5,6){\makebox(0,0){$(1,\frac{1}{3},\frac{1}{3},\frac{1}{3})$}}

\put(15.5,2){\line(3,2){3}}
\put(18.5,4){\line(1,0){1}}
\put(15.5,2){\makebox(0,0){$\bullet$}}
\put(17.5,3){\makebox(0,0){$\bullet$}}
\put(19.5,4){\makebox(0,0){$\bullet$}}
\put(18.5,4){\makebox(0,0){$\bullet$}}
\put (17.5,1){\makebox(0,0){$(\frac{2}{3},\frac{2}{3},\frac{2}{3},0)$}}

\put(23,4){\line(2,1){4}}
\put(23,4){\makebox(0,0){$\bullet$}}
\put(25,5){\makebox(0,0){$\bullet$}}
\put(27,6){\makebox(0,0){$\bullet$}}
\put (25,3){\makebox(0,0){$(\frac{1}{2},\frac{1}{2},\frac{1}{2},\frac{1}{2})$}}

\put(5,5){\vector(1,0){2}}
\put(13,4.3){\vector(2,-1){2}}
\put(13,5.6){\vector(2,1){1.7}}
\put(20,3){\vector(3,1){2}}
\put(20.2,7.3){\vector(3,-2){1.8}}

\end{picture}
\end{center}
\caption{} \label{alpha}
\end{figure}

\begin{definition}
For $[b]\in B(G)$ let $\nu=\nu([b])$ be its Newton point. Let $\N(G)_{\preceq\nu}$ be the image in $\N(G)$ of $\{[b']\in B(G)\mid [b']\preceq[b]\}$. By \cite{Chai}, Prop. 4.4 this only depends on $\nu$. 

For $\nu\in \N(G)$ and $\nu'\in \N(G)_{\preceq\nu}$ a chain between $\nu'$ and $\nu$ is a sequence $\nu'=\nu_0\preceq \dotsm\preceq \nu_n=\nu$ with $\nu_i\in \N(G)_{\preceq\nu}$ and $\nu_i\neq\nu_{i+1}$ for all $i$. We call $n$ the length of the chain. A chain between $\nu'$ and $\nu$  is called maximal if it is not a proper subsequence of another chain between $\nu'$ and $\nu$.
\end{definition}

Before we can compute the lengths of maximal chains we need some more notation
\begin{definition}\label{defdef}
Let $l$ be the number of Galois orbits of absolute fundamental weights of $G$. For $j=1,\dotsc,l$ let $\underline\omega_j$  be the sum of all elements of the corresponding orbit. 

For $b\in G(L)$ the defect of $b$ is defined as $\rk ~G - \rk_{F}J_b$. Here $J_b$ is the reductive group over $F$ with $J_b(A)=\{g\in G(A\otimes_F L)\mid gb=b\sigma(g)\}$ for every $F$-algebra $A$.
\end{definition}
\begin{thm}[Chai, Kottwitz]\label{thmchai}
Let $\nu\in\N(G)_{\mathbb{Z}}$.
\begin{enumerate}
\item The length of chains between two given elements $\nu',\nu''$ with $\nu'\preceq \nu''\in \N(G)_{\preceq\nu}$ is bounded above, in particular there are maximal chains between $\nu'$ and $\nu''$.
  
\item The set $\N(G)_{\preceq\nu}$ is ranked, i.e.~for any $\nu'\preceq \nu''$ in $\N(G)_{\preceq\nu}$ every maximal chain between $\nu'$ and $\nu''$ has the same length, which is independent of $\nu$ and denoted by $\l([\nu',\nu''])$. We have 
$$\l([\nu',\nu''])=\sum_{j=1}^l \lfloor \langle \nu'',\underline{\omega}_j\rangle\rfloor-\lfloor\langle \nu',\underline{\omega}_j\rangle\rfloor.$$
Here $\lfloor x\rfloor$ denotes the greatest integer less or equal to $x$. 
\item If in the above context $\nu''\in X_*(T)$, then  
\begin{eqnarray*}
\l([\nu',\nu''])&=&\langle \rho,\nu''-\nu'\rangle+\frac{1}{2}\defect_G(b')
\end{eqnarray*}
where $[b']\in B(G)$ with $\nu_G([b'])=\nu'$.

\item $\N(G)_{\preceq\nu}$ is a finite set. Every non-empty subset of $\N(G)_{\preceq\nu}$ has a unique supremum and a unique infimum in $\N(G)_{\preceq\nu}$.
\end{enumerate}
\end{thm}
\begin{proof}[References for the proof]
(1) and (4) are due to Chai, \cite{Chai} Thm. 7.4. The formula in (2) is shown essentially by the same proof as given for \cite{Chai}, Theorem 7.4. It corrects the formula in loc. cit. in two respects: One needs to use sums over Galois orbits of absolute fundamental weights, as was pointed out to us by P. Hamacher (compare \cite{Hamacher_Newt}, Prop. 3.11). Furthermore, the combinatorics of counting chain lengths gives the above formula instead of the one of \cite{Chai}, Theorem 7.4. The proof and formula given by Chai are correct for the case $\nu\in X_*(T)_{\dom}$ and then imply our formula for general $\nu$. (3) is due to Kottwitz \cite{Kottwitz2}, and Hamacher \cite{Hamacher_Newt}, 3.3.
\end{proof}

\begin{ex}\label{exlengthgln}
For the group $\GL_n$ there is the following explicit reformulation of the length formula: In this case $l=n$, the $\underline{\omega}_j$ are the fundamental weights $(1,\dotsc,1,0,\dotsc,0)$ with multiplicities $j,n-j$. Thus $\langle \nu,\underline{\omega}_j\rangle$ is the value of the Newton polygon at the point $j$. Assume for the moment that the polygon lies above the first coordinate axis, then $\lfloor\langle \nu,\underline{\omega}_j\rangle\rfloor$ is the number of points of the form $(j,l)$ with $0\leq l\in \mathbb{Z}$ lying on or below $p_{\nu}$. Altogether (and without the additional assumption) we obtain that the length of each maximal chain of Newton points between $\nu'$ and $\nu''$ (with $\nu'\preceq\nu''$) is equal to the number of points with integral coordinates lying on or below $\nu''$ and strictly above $\nu'$.
\end{ex}

\subsection{The Newton stratification}

We now consider stratifications induced by this invariant, or rather by the pair of invariants $(\nu,\kappa_G)$.

Recall the Cartan decomposition $G(L)=\bigcup_{\mu\in X_*(T)_{\dom}}G(\O_L)\mu(\e)G(\O_L)$.

\begin{definition}
For $\mu\in X_*(T)_{\dom}$ let $B(G,\mu)$ denote the set of $[b]\in B(G)$ with $[b]\preceq [\mu(\epsilon)]$.
\end{definition}

\begin{thm}[Katz, Kottwitz]\label{thmmazur} Let $b\in G(L)$ and let  $\mu\in X_*(T)_{\dom}$ with $b\in G(\O_L) \mu(\e)G(\O_L)$. Then $[b]\in B(G,\mu)$.
\end{thm}

\begin{remark}
Therefore for every $\mu$ we obtain a natural map $C(G,\mu)\rightarrow B(G,\mu)$ with $\ll b\rr\mapsto [b]$. Note that $B(G,\mu)$ is not defined as the set of $\sigma$-conjugacy classes of elements of $G(\O_L)\mu(\e)G(\O_L)$. Therefore it is a priori not clear if this map is surjective, this only follows from Theorem \ref{thmne} below.  

The theorem applies in particular to the case where $b\in G(L)$ is a representative of $\Upsilon(x)\in C(G,\mu)$ for some point $x$ in the reduction of a Shimura variety of Hodge type, and $\mu$ is the cocharacter associated with the Shimura variety.

The theorem has originally been proved by Katz \cite{Katz} for $F$-isocrystals (i.e.~the case of $G=\GL_n$) who refers to it as Mazur's inequality. The group theoretic generalization is due to Kottwitz \cite{KottwitzHN}. In both references this is only stated for $F=\Q_p$, but the same proof also shows the analog for function fields.
\end{remark}

We have the following converse.
\begin{thm}[Kottwitz-Rapoport \cite{KR}, Lucarelli \cite{Lucarelli}, Gashi \cite{Gashi}]\label{thmne}
Let $\mu\in X_*(T)_{\dom}$ and let $[b]\in B(G,\mu)$. Then there is an $x\in G(\O_L)\mu(\epsilon) G(\O_L)$ with $[x]=[b]$.
\end{thm}

We now discuss a first important geometric property of Newton strata.

\begin{thm}[Grothendieck, Rapoport-Richartz] \label{thmgrothspec}
Let $S$ be a scheme of characteristic $p$. Let $X$ be an $F$-isocrystal with $G$-structure over $S$. Let $[b]\in B(G)$. Then $$S_{\preceq [b]}(k)=\{s\in S(k)\mid [X_s]\preceq [b]\}$$ defines a closed subscheme $S_{\preceq [b]}$ of $S$, the \emph{closed Newton stratum associated with $[b]$}. The \emph{Newton stratum associated with $[b]$} is the open subscheme of $S_{\preceq [b]}$ defined by $$S_{[b]}(k)=\{s\in S(k)\mid [X_s]= [b]\}.$$
\end{thm}
\begin{remark}\label{remspec}
\begin{enumerate}
\item Here an $F$-isocrystal with $G$-structure is an exact faithful tensor functor $({\rm Rep}_{\Q_p}G)\rightarrow (F-{\rm Isoc})$, compare \cite{RapoportRichartz}, Def. 3.3. In particular this theorem applies to the reduction modulo $p$ of Shimura varieties of Hodge type.

\item This theorem is known as Grothendieck's specialization theorem. Grothendieck's original proof \cite{Gr} is for isocrystals (or $p$-divisible groups) without additional structure. The group theoretic generalization is due to Rapoport and Richartz, \cite{RapoportRichartz}. 

\item Assume that the map $\kappa_G$ has the constant value $\kappa_G([b])$ in all closed points of $S$. This is for example the case if $S$ is connected. Then for $\nu=\nu([b])$ we also write $\N_{\nu}$ instead of $S_{[b]}$ (and similarly for $S_{\preceq [b]}$).

\item There is the following analog in the function field case. Let $LG$ be the loop group of $G$, i.e.~the ind-scheme over $\F_q=\O_F/\e\O_F$ representing the functor assigning to an $\F_q$-algebra $R$ the set $G(R\ll \e\rr[\frac{1}{\e}])$. Let $X\in LG(S)$ for $S$ as in the theorem. Then the same statement as in Theorem \ref{thmgrothspec} holds for the induced decomposition of $S$. This analog follows from essentially the same proof as Theorem \ref{thmgrothspec}, compare also \cite{HV1}, Theorem 7.3.
\end{enumerate}
\end{remark}
\begin{definition}
Let $\S_K(G,X)$ be a Hodge type Shimura variety and let $\mu$ be the associated cocharacter. Then the $\mu$-ordinary Newton stratum is the one associated with the unique maximal element $[\mu(p)]$ of $B(G,\mu)$  (compare Theorem \ref{thmchai} (4)). It is open in $\S_K(G,X)$. The basic Newton stratum is the one associated with the unique minimal element of $B(G,\mu)$. It is closed in $\S_K(G,X)$. 
\end{definition}

In the following sections we explain converses of the two preceding theorems.

\section{Non-emptiness of strata}

By Theorem \ref{thmmazur} the set of Newton strata in a Shimura variety of Hodge type with hyperspecial level structure is indexed by the set $B(G,\mu)$ where $G$ and $\mu$ are the group and coweight determined by the Shimura datum. However, it is a priori not clear if for each $[b]\in B(G,\mu)$ the corresponding stratum $\N_{[b]}\subseteq \S_K(G,X)$ is indeed non-empty. To show such a statement, the task is to construct an abelian variety with additional structure whose rational Dieudonn\'e module (with additional structure) is a given one.\\ 

The first result in this direction valid for a large class of Shimura varieties is the following theorem.
\begin{thm}[Viehmann-Wedhorn \cite{vw}, Thm. 1.6]\label{thmnonempty}
Let $G,\mu$ be associated with a Shimura datum of PEL type. Let $[b]\in B(G,\mu)$. Then the Newton stratum associated with $[b]$ in the special fiber of the associated Shimura variety is non-empty.
\end{thm}
\begin{remark}
This proves a conjecture of Rapoport, \cite{RapoportGuide}, Conj. 7.1. Note that non-emptiness of Newton strata does not follow directly from Theorem \ref{thmne}, as we do not know if the map $\Upsilon: \S_K(G,X)(k)\rightarrow C(G,\mu)$ is surjective. Also, it cannot be directly deduced from Kottwitz's \cite{Kottwitzmodp} resp. Kisin's \cite{Kisin} computation of the $\zeta$-function of the Shimura variety. It is a priori not clear from that formula that the contribution of a given Newton stratum is non-zero.

There are several special cases where non-emptiness of Newton strata has been shown already some time ago.

In \cite{twthesis}, Wedhorn shows that for PEL type Shimura varieties (as usual assumed to have good reduction) the $\mu$-ordinary Newton stratum is dense in the Shimura variety and in particular non-empty. Using Kottwitz's description of the points in the good reduction of Shimura varieties of PEL type one can show that the basic Newton stratum is also non-empty (see Fargues \cite{Fargues}, 3.1.8 or Kottwitz \cite{Kottwitzmodp}, 18). 

Using a different language, the non-emptiness question has also been adressed by Vasiu in \cite{Vasiu}.
\end{remark}

\begin{proof}[Strategy of proof of Theorem \ref{thmnonempty}] In addition to the Newton stratification we consider the Ekedahl-Oort stratification. The Eke\-dahl-Oort invariant associates with an abelian variety with  PEL structure its $p$-torsion (with induced additional structure). Group-theoretically, and on geometric points in the special fiber of the Shimura variety, this means the following: Let $x\in \S_K(G,X)(k)$ for some algebraically closed field $k$. Let $\ll g_x\rr=\Upsilon(x)\in C(G,\mu)$. Let $G_1=\{g\in G(\O_L)\mid g\equiv 1 \text{ in }G(k)\}$. Then the Ekedahl-Oort invariant corresponds to considering the $G_1$-double coset of $\ll g_x\rr$ (see \cite{trunc1}). Note that this is well-defined as $G_1$ is a normal subgroup of $G(\O_L)$. The set of possible values of this invariant is finite, and in bijection with a certain subset$~^{\mu}W$ of the Weyl group of $G$. The Ekedahl-Oort invariant induces a stratification of the Shimura variety.

The above definition of the Ekedahl-Oort invariant has a strong focus on the group theoretic background. There are alternative definitions, for example using $G$-zips \cite{MW} that rather use the $p$-divisible groups directly.

The first main ingredient in the proof of Theorem \ref{thmnonempty} is to show that for every element $[b]\in B(G,\mu)$ there is an element of$~^{\mu}W$ such that the corresponding Ekedahl-Oort stratum is contained in the Newton stratum of $[b]$. Strictly speaking, in \cite{vw} we showed a slightly weaker statement by modifying the Shimura datum. However, in the meantime, Nie \cite{Nie} has shown that also the stronger statement above holds, which leads to a simplification of the proof. It is therefore enough to show that all Ekedahl-Oort strata are non-empty.

The second main ingredient is to use flatness of the morphism from $\S_K(G,X)$ to the stack of Ekedahl-Oort invariants (or equivalently: $G$-zips) to show that all Ekedahl-Oort strata are non-empty if and only if the same holds for the one corresponding to the unique closed point in this stack. This `minimal' Ekedahl-Oort stratum is known to be contained in the basic Newton stratum of the Shimura variety.

Finally, one uses that non-emptiness of the basic Newton stratum implies that also the minimal Ekedahl-Oort stratum is non-empty to complete the proof.
\end{proof}
Let $\A$ be an abelian variety over an algebraically closed field of characteristic $p$, let $\A[p^{\infty}]$ be its $p$-divisible group, and $X$ a $p$-divisible group isogenous to $\A[p^{\infty}]$. Dividing by the kernel of such an isogeny one can directly construct an abelian variety $\mathcal B$ isogenous to $\A$ and with $\mathcal B[p^{\infty}]=X$. A version of this argument including the PEL structure (see \cite{vw}, 11) leads to the following integral version of Theorem \ref{thmnonempty}.

\begin{thm}[\cite{vw}, Theorem 1.6(2)]\label{thmintnonempty}
Let $\D$ be a PEL Shimura-datum and let $\S_K(G,X)$ be the corresponding Shimura variety. Then for every $p$-divisible group $\mathcal G$ with $\D$-structure over an algebraically closed field $k$ in characteristic $p$ there is a $k$-valued point of $\S_K(G,X)$ whose attached $p$-divisible group with $\D$-structure is isomorphic to $\mathcal G$. In group-theoretic terms: The map $\S_K(G,X)(k)\rightarrow C(G,\mu)$ is surjective.
\end{thm}

\begin{remark}
Since the proof of the above Theorem \ref{thmnonempty}, the question has attracted significant attention. Several people have generalized Theorem \ref{thmnonempty} further to other classes of Shimura varieties of Hodge type, and/or have given other strategies to prove non-emptiness of Newton strata. 
\begin{enumerate}
\item Scholze and Shin \cite{SS}, Corollary 8.4 give a different proof of Theorem \ref{thmnonempty} for certain compact unitary group Shimura varieties. They use the description of points on the Shimura variety by Kottwitz triples and stabilization of the trace formula to show non-emptiness in this case.
\item Kret \cite{Kret} uses Kisin's formula for the number of mod $p$ points of Shimura varieties of Hodge type (see \cite{Kisin}), together with stabilization of the twisted trace formula. In this way he shows non-emptiness of Newton strata for Shimura varieties of Hodge type where the associated group is of type (A), and under assumption of the stabilization of the trace formula also for types (B) and (C).
\item Koskivirta \cite{Koskivirta} uses generalized Hasse invariants to show that if $\S_K(G,X)$ is a Shimura variety of Hodge type that is projective, then all Ekedahl-Oort strata are non-empty. Using a result of Nie \cite{Nie} he can then in the same way as in the original proof of Theorem \ref{thmnonempty} deduce that also all Newton strata are non-empty. 
\item Lee gives two proofs of the generalization of Theorem \ref{thmnonempty} to Shimura varieties of Hodge type, see \cite{Lee}. In his first proof he constructs a point in the given Newton stratum by realizing it as the image of a point of a special sub-Shimura variety. For the second proof he uses that by work of Kisin \cite{Kisin}, the points mod $p$ of a Shimura variety of Hodge type are in bijection with Kottwitz triples. The Newton point can be read off easily from one of the components of the Kottwitz triple. Lee's approach is then to show directly that for each $\sigma$-conjugacy class a suitable Kottwitz triple exists.
\item C.-F. Yu \cite{Yutalk} announced a proof of non-emptiness of the basic locus of Shimura varieties having what he calls an enhanced integral model. This kind of integral model is defined  by an abstract set of axioms, for example satisfied by the integral models for Shimura varieties of Hodge type. He then also uses the strategy of constructing points via special sub-Shimura varieties, using a Lemma shown by Langlands and Rapoport.
\item Kisin \cite{Kisintalk} has announced a proof (in joint work with Madapusi Pera and Shin) for all Shimura varieties of Hodge type, also using special sub-Shimura varieties, and the Langlands-Rapoport lemma. He even only needs to assume that $G$ is quasi-split instead of unramified.
\end{enumerate}
\end{remark}

\section{Closure relations and dimensions}
\subsection{The Grothendieck conjecture on closures of Newton strata}
In 1970 (\cite{Gr}, letter to Barsotti, p. 150), Grothendieck conjectured the following converse to his specialization theorem (compare Theorem \ref{thmgrothspec}).

\begin{conj}
Let $\mathbb X_0$ be a $p$-divisible group over a field $k$ of characteristic $p$, and let $\nu$ be its Newton polygon. Let $\nu'\in \N(G)_{\Z}$ with $\nu\preceq \nu'\preceq \mu$. Here $\mu=(1,\dotsc,1,0,\dotsc,0)$ with multiplicities the codimension and dimension of $\mathbb X_0$ is the Hodge polygon of $\mathbb X_0$. Then there is a $p$-divisible group $\mathbb{X}$ over $\Spec(k[[t]])$ with special fiber $\mathbb X_0$ and such that the generic fiber has Newton polygon $\nu'$.
\end{conj}
This conjecture has been shown by Oort in 2002, see Remark \ref{remoort}(1) below.

Using the bijection between deformations of an abelian variety in characteristic $p$ and deformations of its $p$-divisible group, (as well as liftability properties in the spirit of Theorem \ref{thmintnonempty}) one can translate this into a statement on the closures of Newton strata in a suitable Shimura variety. The natural generalization to Shimura varieties of PEL type is the following theorem.
\begin{thm}[Hamacher]\label{thmgroth}
Let $\S_K(G,X)$ be a Shimura variety of PEL type with good reduction at $p$. Let $[b]\in B(G,\mu)$ where $\mu$ is the cocharacter associated with the Shimura datum, and let $\nu=\nu([b])$. Then
$$\overline{\N_{\nu}}=\bigcup_{\nu'\preceq \nu} \N_{\nu'}.$$ 
\end{thm}

\begin{remark}\label{remoort}
Before Hamacher's proof, a number of partial results were already known.
\begin{enumerate}
\item Grothendieck's original conjecture corresponds to the case that $G_{\mathbb{Q}_p}\cong \GL_h$. This has been shown by Oort \cite{Oort1}. There, Oort also proved Theorem \ref{thmgroth} for the Siegel modular variety (i.e. for $G=\GSp_{2n}$), or equivalently for deformations of principally polarized $p$-divisible groups.
\item Wedhorn used explicit deformations to prove in \cite{twthesis} that the $\mu$-ordinary Newton stratum (i.e.~the one corresponding to the unique maximal element of $B(G,\mu)$ with respect to $\preceq$) of a Shimura variety of PEL type with good reduction is dense. This was generalized to Shimura varieties of Hodge type by Wortmann \cite{Wortmann}, using Ekedahl-Oort strata.
\item The general statement of Theorem \ref{thmgroth} (together with weaker forms) was asked for by Rapoport in \cite{RapoportGuide}, Questions 7.3.
\item In a small number of other special cases, this statement was known due to explicit calculations, for example for the Shimura variety studied by Harris and Taylor in their proof of the local Langlands correspondence for $\GL_n$.
\end{enumerate}
\end{remark}

\begin{remark}
Parts of Oort's proof of the original Grothendieck conjecture have also been generalized, however, in general this has not led to a proof of Theorem \ref{thmgroth}. Let us recall briefly the two main steps in Oort's approach. The first step is to show that one can deform a given $p$-divisible group to a $p$-divisible group with the same Newton point, but $a$-invariant $\leq 1$. Here, the $a$-invariant of a $p$-divisible group $\mathbb{X}$ over a field $k$ is defined as $\dim_k\Hom(\alpha_p,\mathbb{X})$. The main difficulty is here to consider the case where the isocrystal of the $p$-divisible group is simple (and then to use induction on the number of simple summands for the general case). Main ingredients in this step of the proof (in joint work of de Jong and Oort, \cite{JO}) are de Jong and Oort's purity theorem for the Newton stratification (a weak form of Theorem \ref{thmpurity} below), and combinatorial arguments on some stratification of the Rapoport-Zink moduli space of $p$-divisible groups isogenous to $\mathbb{X}$. A key intermediate result is that the Rapoport-Zink moduli space is equi-dimensional, and to compute its dimension. 

In a second step Oort uses a Cayley-Hamilton type argument to explicitly compute the Newton polygon of deformations of $p$-divisible groups of $a$-invariant 1 to conclude. This second part of the argument has been generalized by Yu \cite{Yu}.

However, it seems to be hard to extend this to a proof of the whole statement, simply because there is (for this particular purpose at least) no good analogue of the $a$-invariant on moduli spaces of $p$-divisible groups with additional structure. For general Newton strata on Shimura varieties of PEL type the locus where the $a$-invariant is at most $1$ is in general no longer dense, it can even be empty. 
\end{remark}


\subsection{An analog for function fields}
The proof of Theorem \ref{thmgroth} is modelled along the lines of an analogous statement for the function field case, i.e. for Newton strata in loop groups \cite{grothconj}. 

To formulate this statement let $G$ be a split connected reductive group over $\mathbb{F}_q$, and let $B\supseteq T$ be a Borel subgroup and a split maximal torus of $G$. We denote by $LG$ the loop group of $G$, i.e.~the ind-scheme over $\mathbb{F}_q$ representing the sheaf of groups for the fpqc-topology whose sections for an $\mathbb{F}_q$-algebra $R$ are given by $LG(R)=G(R\dl z\dr )$, see \cite{Faltings}, Definition 1. Let $LG^+$ be the sub-group scheme of $LG$ with $LG^+(R)=G(R[[z]])$. 

Let $S$ be a scheme and $Y\in LG(S)$. Recall from Remark \ref{remspec}(3) that for $[b]\in B(G)$  the locus $\mathcal{N}_{[b]}\subseteq S$ where $Y$ is in $[b]$ defines a locally closed reduced subscheme of $S$. 

On the set of dominant coweights $X_*(T)_{\dom}$ we consider the partial ordering induced by the Bruhat order, i.e.~$\mu\preceq\mu'$ if and only if $\mu'-\mu$ is a non-negative integral linear combination of positive coroots. Note that this is slightly finer than the ordering induced by $\preceq$ on $X_*(T)_{\Q,\dom}$, as we do not allow rational linear combinations.

\begin{thm}\label{thm2}
Let $\mu_1\preceq\mu_2\in X_*(T)$ be dominant coweights. Let $$S_{\mu_1,\mu_2}=\bigcup_{\mu_1\preceq\mu'\preceq \mu_2} LG^+\mu'(z)LG^+.$$ Let $[b]\in B(G,\mu_2)$. Then the Newton stratum $\mathcal{N}_{[b]}=[b]\cap S_{\mu_1,\mu_2}$ is non-empty and pure of codimension $$\l([\nu([b]),\mu_2])=\langle \rho,\mu_2-\nu([b])\rangle+\frac{1}{2}\defect(b)$$ in $S_{\mu_1,\mu_2}$. The closure of $\mathcal{N}_{[b]}$ in $S_{\mu_1,\mu_2}$ is the union of all $\mathcal{N}_{[b']}$ for $[b']\preceq [b]$.
\end{thm}
Here the defect $\defect(b)$ is as in Definition \ref{defdef}  defined as $\rk ~G - \rk_{\mathbb{F}_q\dl z\dr}J_b$.

Note that one still needs to define the notions of codimension and of the closure of the infinite-dimensional schemes $\mathcal{N}_b$. Both of these definitions use that there is an open subgroup $H$ of $LG^+$ such that the Newton point of an element of $S_{\mu_1,\mu_2}$ only depends on its image in the finite-dimensional scheme $S_{\mu_1,\mu_2}/H$, compare \cite{grothconj}, 4.3.

\begin{remark}
Theorem \ref{thm2} is not precisely the analog of Theorem \ref{thmgroth}:
\begin{enumerate}
\item In this context we can consider the Newton stratification on arbitrary (unions of) double cosets. The double coset is the replacement of the Hodge polygon in the Shimura variety context. There, one is limited to the case of minuscule $\mu$. 
\item For the closure relations we could also consider the Newton stratification on the whole loop group, but this yields a much weaker statement: There it is very easy to see that the closure of a stratum is a union of strata (just because the $\sigma$-conjugation action of $G(L)$ is transitive on each Newton stratum), and it consists of the conjectured strata because of \cite{vw}, Cor. 1.9.
There is no good analogue for the whole loop group of the statement about codimensions.
\item Theorem \ref{thm2} is only proven for split reductive groups. The reason is that at the time, the dimension formula and theory of local $G$-shtukas that is used in the proof was not yet developed in greater generality. We expect that essentially the same proof shows the theorem for all unramified groups. (Which is also justified by the fact that the analog for Shimura varieties in Theorem \ref{thmgroth} is shown in that generality.) On the other hand, the assumption that $K_p$ is a hyperspecial maximal compact subgroup cannot easily be removed.
\item The direct analog of Theorem \ref{thmgroth} in the function field context would be a statement about moduli spaces of shtukas with additional structure. The present statement is rather the generalized analog of Grothendieck's conjecture in the sense that it considers deformations of the local data. In the same way as for abelian varieties and $p$-divisible groups it should be possible to use the correspondence between deformations of global shtukas and the corresponding local shtukas to translate Theorem \ref{thm2} into a statement about moduli spaces of global shtukas with additional structure.
\end{enumerate}
\end{remark}

\subsection{Outline of the proof and dimensions of Newton strata}
We will spend the rest of this section to outline some of the joint main ingredients of the proofs of Theorems \ref{thmgroth} and \ref{thm2}. For simplicity we formulate statements only in the context of Shimura varieties. For the corresponding statements in the function field context compare \cite{grothconj}. The strategy of proof differs from Oort's proof of the Grothendieck conjecture in the sense that it avoids the use of the $a$-invariant. Instead we use a stronger version of the purity theorem to derive the closure relations directly from the computation of dimensions of suitable moduli spaces.

The proof of Theorem \ref{thmgroth} at the same time proves a formula for the dimensions of Newton strata which is interesting in its own right. It is the counterpart of the statement about codimensions in Theorem \ref{thm2}.

\begin{thm}[Hamacher]\label{thmdim}
Let $\S_K(G,X)$ be a Shimura variety of PEL type, and let $[b]\in B(G,\mu)$ be in the associated set of $\sigma$-conjugacy classes. Then the Newton stratum $\N_{[b]}$ in  $\S_K(G,X)$ is equidimensional of dimension $$\langle \rho, \mu+\nu([b])\rangle-\frac{1}{2}\defect(b).$$
\end{thm}

The translation between statements about (co)dimensions and statements about closure relations is done using purity properties. The following definition and Lemma \ref{lemtranslate} are an abstract version of the key idea behind the proofs of \ref{thm2}, \ref{thmgroth} and \ref{thmdim}.
\begin{definition}
Let $G$ be a reductive group over $\F_q$, $B$ a Borel subgroup, $T\subset B$ a maximal torus, and let $\mu\in X_*(T)_{\dom}$ be given.
Let $S$ be a scheme, and for every $[b]\in B(G,\mu)$ let $S_{[b]}$ be a locally closed subscheme such that $S$ is the disjoint union of the $S_{[b]}$ and for every $[b]$, the subscheme $\bigcup_{[b']\preceq[b]}S_{[b']}$ is closed. Then we say that this decomposition of $S$ satisfies strong purity if the following condition holds. Let $[b]\in B(G,\mu)$ such that $S_{[b]}\neq\emptyset$. Let $[b']\preceq [b]$ such that $\l([\nu[b'],\nu( [b])])=1$. Let $I$ be the closure in $S$ of an irreducible component of $S_{[b]}$. Let $I^{[b']}$ be the complement in $I$ of $I\cap \bigcup_{[b'']\preceq[b' ]}S_{[b'']}$. Then for every such $I$ and $[b']$ we have that $I^{[b']}$ is an affine $I$-scheme.

The decomposition satisfies weak purity if for every $I$ as above, $I\cap \N_{[b]}$ is an affine $I$-scheme.
\end{definition}
\begin{ex}
For the group $\GL_n$ we have the following explicit description: Let $\nu,\nu'$ be the Newton points of $[b]$ and $[b']$ with $[b],[b']$ as in the definition above. Then the integral points lying on or below $p_{\nu}$ are the same as those for $p_{\nu'}$ except for one breakpoint $y$ of $\nu$ which does no longer lie on $\nu'$. The subscheme $I^{[b']}$ is now the union of all Newton strata of $I$ where the Newton polygon (which automatically lies below $\nu$) contains the given point $y$.
\end{ex}

\begin{thm}\label{thmpurity}
The Newton stratification on the reduction modulo $p$ of PEL Shimura varieties satisfies strong purity.
\end{thm}
\begin{remark}
The first version of a theorem along these lines was shown by de Jong and Oort \cite{JO}, who showed that the Newton stratification for $p$-divisible groups without additional structure satisfies a weaker version of weak purity where instead of affineness we require that the complement of $I\cap \N_{[b]}$ in $I$ is either empty or pure of codimension 1. Vasiu \cite{Vasiu1} showed that this same stratification also satisfies weak purity as defined above. Finally, Hartl and the author proved a group-theoretic version of weak purity in the function field case (i.e. for moduli of locl $G$-shtukas), see \cite{HV1}. 

Only several years later, Yang \cite{yang} introduced the concept of strong purity (under a different name) and observed that de Jong and Oort's statement can be generalized to prove (the same weakened version of) strong purity for the Newton stratification on families of $p$-divisible groups. After this idea, the same generalizations as for weak purity were made, in the function field case in \cite{grothconj}, for Shimura varieties of PEL type in \cite{Hamacher_Newt}.
\end{remark}

\begin{lemma}\label{lemtranslate}
Let $S$ be an irreducible scheme and let $S=\bigcup_{[b]\in B(G,\mu)}S_{[b]}$ be a decomposition satisfying strong purity. Let $[b_{\eta}]\in B(G,\mu)$ be the generic $\sigma$-conjugacy class. Let $[b_0]\in B(G,\mu)$ be such that $S_{[b_0]}$ is non-empty and that the codimension of every irreducible component of $S_{[b_0]}$ in $S$ is at least $\l([\nu([b_0]),\nu([b_{\eta}])])$. Then for every $[b']\in B(G,\mu)$ with $\nu([b_0])\preceq\nu([b'])\preceq \nu([b_{\eta}])$ we have
\begin{enumerate}
\item $S_{[b_0]}\subset \overline{S_{[b']}}.$
\item In particular, $S_{[b']}\neq \emptyset$.
\item Every irreducible component $I$ of $S_{[b_0]}$ is of codimension $\l([\nu([b_0]),\nu([b_{\eta}])])$ in $S$.
\end{enumerate}
\end{lemma}
\begin{proof}
The proof is based on the fact that if $I^{[b']}$ is an affine $I$-scheme, then its complement in $I$ is either empty or pure of codimension 1. Using that, the lemma is shown in the same combinatorial way as the more explicit statements for Newton strata in loop groups \cite{grothconj} or Shimura varieties \cite{Hamacher_Newt}, Prop. 5.13.
\end{proof}

In view of Lemma \ref{lemtranslate}, Theorem \ref{thmpurity} implies that in order to prove Theorems \ref{thmgroth} and \ref{thmdim} it is enough to show that for every element in $B(G,\mu)$, the codimension of each irreducible component of the corresponding Newton stratum is at least $\l([\nu([b]),\mu])$.

From the almost product structure on Newton strata (see \cite{OortFol0} for the case of $G=\GL_n$ or $\GSp_{2n}$, \cite{MantThesis}, 4 for the general case, or \cite{Hamacher_Newt}, Prop. 6.4) we obtain for every Newton stratum $\N_{\nu}$ in $\S_K(G,X)$ with $\nu=\nu([b])$ associated with a $p$-divisible group with additional structure $\underline{\mathbb{X}}$ that
$$\dim~ \N_{\nu}=\dim~X_{\mu}(b)+\dim ~C_{\underline{\mathbb{X}}}.$$ 
Here, $C_{\underline{\mathbb{X}}}$ is the central leaf associated with $\underline{\mathbb{X}}$. The first summand is the dimension of the Rapoport-Zink space associated with $\underline{\mathbb{X}}$ or equivalently of the affine Deligne-Lusztig variety $X_{\mu}(b)$. Note that central leaves and the above decomposition are essentially a special case of the foliation structure on Newton strata using Igusa varieties explained in \cite{Mant}, 5.

For the two summands we have
\begin{thm}[\cite{GHKR},\cite{dimdlv},\cite{HamADLV}, Thm. 1.1]
Let $[b]\in B(G,\mu)$. Then $$\dim~ X_{\mu}(b)=\langle\rho,\mu-\nu\rangle-\frac{1}{2}\defect(b).$$
\end{thm}
Another proof of this result for unramified groups $G$ in the arithmetic context has recently been given by Zhu \cite{Zhu},3.
\begin{thm}[\cite{OortFol}, \cite{Hamacher_Newt}, Cor. 7.8]
Let $[b]\in B(G,\mu)$ be the class corresponding to $\underline{\mathbb{X}}$. Then $$\dim ~C_\mathbb{X}=\langle 2\rho,\nu([b])\rangle.$$
\end{thm}
Altogether we obtain
\begin{eqnarray*}
\dim~ \N_{\nu}&=&\dim~X_{\mu}(b)+\dim ~C_{\underline{\mathbb{X}}}\\
&=&\langle\rho,\mu+\nu\rangle-\frac{1}{2}\defect(b)\\
&=&\dim~ \S_K(G,X)- \langle\rho,\mu-\nu\rangle-\frac{1}{2}\defect(b)\\
&=&\dim ~\S_K(G,X)- \l([\nu,\mu]).
\end{eqnarray*}
Thus the codimension of every irreducible component of $\mathcal{N}_{\nu}$ is at least $\l([\nu,\mu])$, and Lemma \ref{lemtranslate} finishes the proof of Theorems \ref{thmgroth} and \ref{thmdim}.

~\\
\noindent{\it Acknowledgments.} I thank P. Hamacher for helpful discussions. The author was partially supported by ERC starting grant 277889 ``Moduli spaces of local $G$-shtukas''.

\end{document}